\documentclass[12pt, a4paper]{amsart}
\usepackage{amsmath}
\usepackage{amssymb}
\usepackage[mathcal]{eucal}
\usepackage[all]{xy}
\usepackage{latexsym}
\usepackage{amstext}
\usepackage{amsfonts}
\usepackage{amsthm}
\usepackage{amsopn}
\usepackage{amsbsy}
\usepackage{layout}
\usepackage{color}
\usepackage{graphicx}
\usepackage{bm}

\theoremstyle{plain}
\newtheorem{theorem}{Theorem}
\newtheorem{proposition}[theorem]{Proposition}

\newtheorem{corollary}[theorem]{Corollary}

\theoremstyle{definition}

\newtheorem{problem}[theorem]{Problem}
\theoremstyle{remark}
\newtheorem{fact}[theorem]{Fact}
\newtheorem{remark}[theorem]{Remark}
\newtheorem{example}[theorem]{Example}

\begin{document}
\title[Lotz-Peck-Porta and Rosenthal's theorems]{Lotz-Peck-Porta and Rosenthal's theorems for spaces $C_p(X)$}
\author{Jerzy K{\c{a}}kol}
\address{Faculty of Mathematics and Informatics. A. Mickiewicz University,
61-614 Pozna\'{n}}
\email{kakol@amu.edu.pl}
\author{Ond\v {r}ej Kurka}
\address{Institute of Mathematics, Czech Academy of Sciences, Prague, Czech Republic}
\email{kurka.ondrej@seznam.cz}
\author{Wies{\l}aw \'Sliwa}
\address{Faculty of Exact and Technical Sciences,
University of Rzesz\'ow, 35-310 Rzesz\'ow}
\email{sliwa@amu.edu.pl; wsliwa@ur.edu.pl}

\thanks{The authors thank Prof.~Arkady Leiderman and Prof.~Witold Marciszewski for their discussions and helpful comments, which allowed the use of Theorem 2 in Rosenthal's publication \cite{Rosenthal}. The second named author was supported by the Academy of Sciences of the Czech Republic (RVO 67985840).}

\begin{abstract}
For a Tychonoff space $X$ by $C_p(X)$  we denote the space $C(X)$ of continuous real valued functions on $X$ endowed with the pointwise topology. We prove that an infinite compact space $X$ is scattered if and only if every closed infinite-dimensional subspace in $C_p(X)$ contains a copy of $c_0$ (with the pointwise topology) which is complemented in the whole space $C_p(X)$. This provides a $C_p$-version of the  theorem of Lotz, Peck and Porta for Banach spaces $C(X)$ and $c_0$. Applications will be provided.  We prove   also a $C_p$-version of Rosenthal's theorem by showing that for an infinite compact $X$ the space $C_p(X)$ contains a closed copy of $c_{0}(\Gamma)$ (with the pointwise topology) for some uncountable set $\Gamma$ if and only if $X$ admits  an uncountable family of pairwise disjoint open subsets of $X$.
 Illustrating examples, additional supplementing $C_p$-theorems  and comments are included.
\end{abstract}
\maketitle

\section{Introduction}
For  a locally convex space (lcs) $E$  by $E'$ we denote its topological dual, especially $E'_{\beta}=(E',\beta(E',E))$ means  the strong dual of $E$.
For a Tychonoff space $X$ by $C_p(X)$  we denote the space $C(X)$ of continuous real valued functions on $X$ endowed with the pointwise topology. By $c_0$ we denote the classical Banach space of real sequences $x=(x_{n})$ converging to zero with the sup-norm topology.

In 1958 Pe{\l}czy\'nski and Semadeni  \cite[Main Therorem]{Pelczynski} proved a remarkable  theorem gathering several equivalent conditions characterizing scattered compact spaces. Among the others they showed that a compact space $X$ is scattered if and only if every closed infinite-dimensional closed subspace of $C(X)$ contains an isomorphic copy of the Banach sequence space $c_0$.  In  \cite[Theorem 11]{LPP} Lotz, Peck and  Porta extended this theorem, see also \cite{LPP-1} for several  concrete examples semi-embeddings between Banach spaces.
\begin{theorem}[Lotz--Peck--Porta]\label{PS} For an infinite compact space $X$ the following assertions are equivalent:
\begin{enumerate}
\item $X$ is scattered.
\item Any closed infinite-dimensional vector subspace  of the Banach space $C(X)$ contains a copy of the  Banach space $c_0$ which is complemented in $C(X)$.
\item  Any semi-embedding from  $C(X)$ into a Banach space is embedding.
\item If  $E$  is a Banach space and $T: C(X)\rightarrow E$  is an injective linear continuous map which is not
 an embedding, then there is a complemented subspace $G$ of $C(X)$ such that $G$ is
 isomorphic to $c_0$ and $T|G$ is compact.
\end{enumerate}
\end{theorem}
The term semi-embedding denotes an injective continuous linear map  from one
 Banach space into another, which maps the closed unit ball of the domain onto a
 closed set. An embedding of a Banach space into
 another is a continuous linear map  which is an isomorphism onto a closed subspace of its
 codomain.

 If $E$ and $F$ are lcs  and $T:E\rightarrow F$ is a continuous linear map, we shall say  that  $T$ is \emph{polar} if there exists in $E$  a base $\mathcal{U}$ of closed neighbourhoods of zero such that $T(U)$ is closed in $F$ for each $U\in\mathcal{U}$. If additionally $T$ is injective, we  say that $T$ is \emph{semi-embedding}.

In \cite[Theorem 6]{Kakol-Kurka} we proved another alternative characterization of compact scattered spaces as follows:
\begin{theorem}[K\c akol--Kurka]\label{Main2}
Let $X$ be an infinite compact space.  The following assertions are equivalent:
\begin{enumerate}
\item $X$ is scattered.
\item There is no infinite-dimensional closed $\sigma$-compact vector subspace of $C_p(X)$.
\item There is no infinite-dimensional vector subspace $F$ of $C_p(X)$ admitting  a fundamental sequence of bounded sets.
\end{enumerate}
\end{theorem}
In the paper we show the following $C_p$-variant of Theorem \ref{PS} which  will also provide a more comprehensive look at the Lotz, Peck and Porta result.
\begin{theorem}\label{MAIN}
For an infinite compact space $X$  the following assertions are equivalent:
\begin{enumerate}
\item $X$ is scattered.
\item Every closed infinite-dimensional vector subspace of $C_p(X)$ contains a copy of $(c_0)_p$ which is complemented in the space $C_p(X)$, where $(c_0)_p=\{(x_{n})\in\mathbb{R}^{\mathbb{N}}: x_{n}\rightarrow 0\}$ is endowed with the topology of $\mathbb{R}^{\mathbb{N}}$.
\end{enumerate}
\end{theorem}
Note  $(c_0)_p \approx C_p(S)$, where $S=\{n^{-1}:n\in\mathbb{N}\}\cup \{0\}$.
\begin{remark}
If $X$ is infinite compact  and scattered, our Theorem \ref{MAIN}  provides a fulfilling perspective on the Lotz, Peck, Porta theorem. Indeed, if $E$ is any infinite-dimensional closed vector subspace of  $C(X)$, Theorem \ref{MAIN}  and the classical closed graph theorem yield   that $E$ contains not only a copy of the Banach space $c_0$ complemented  in the Banach space $C(X)$ (with a continuous projection, say
$T:C(X)\rightarrow c_0$) but simultaneously   contains a copy  of the space $(c_0)_p$  complemented  in $C_p(X)$  with the same continuous projection $T: C_p(X)\rightarrow (c_0)_p$, see Proposition \ref{0}.
\end{remark}
On the other hand, we prove also the following $C_p$-version of Rosenthal's \cite[Theorem 4.5]{Rosenthal}.
\begin{theorem}\label{MAIN3} For a compact space $X$ the following  are equivalent.
\begin{enumerate}
\item  There exists an uncountable family of pairwise disjoint open subsets of $X$.

\item  $C_p(X)$ contains a copy of $(c_{00}(\Gamma))_p$ for some uncountable set $\Gamma$.

\item $C_p(X)$ contains a closed copy of $(c_{0}(\Gamma))_p$ for some uncountable set $\Gamma$.

\item $C(X)$ contains a copy of $c_{0}(\Gamma)$ for some uncountable set $\Gamma$.
\item $C_p(X)$ contains a compact non-separable subset.
\end{enumerate}
\end{theorem}
Last Theorem \ref{MAIN3} yields  the following
\begin{corollary}\label{MAIN2}
A compact scattered space $X$ is not separable if and only if $C_p(X)$ contains a [closed] copy of $(c_0(\Gamma))_p$  for an uncountable set  $\Gamma$.
\end{corollary}
Using results from  \cite[Example 2.16]{Dow} and \cite[Proposition 4.2]{Galego} and Corollary  \ref{MAIN2} we have however that
\begin{example}\label{exam7}
There exists a compact scattered non-separable $X$ such that $C_p(X)$ contains a closed copy of $(c_0(\aleph_1))_p$ not complemented in $C_p(X)$ but $(c_0(\aleph_1))_p$ contains  a copy of $(c_0)_p$ complemented in $C_p(X)$.
\end{example}
Theorem \ref{PS}(3) and Theorem \ref{MAIN}   may suggest the following
\begin{problem}\label{problem1}\label{U}
Describe suitable  Tychonoff infinite spaces $X$ and those lcs $E$ such that every   semi-embedding  from $C_p(X)$ into  $E$ is  embedding, i.e. an isomorphism onto the range.
\end{problem}
However, it seems  to be unclear how to construct semi-embeddings of $C_p(X)$ into $C_p(Y)$ which are  not  embeddings. For special $X$ and $Y$ and injective maps $T$ we have the following simple and easily seen
\begin{example}\label{example1}
If $X$ is an infinite Tychonoff space and $Y$ is its dense proper subspace, the restriction map $T: C_p(X)\rightarrow C_p(Y)$ is injective but not semi-embedding.
\end{example}
On the other hand, in \cite{LLP} Leiderman, Levin and Pestov showed that in general, linear continuous surjections
of $C_p$-spaces fail to be open. In fact they proved \cite[Theorem 1.8]{LLP} that there  exists a linear continuous surjection of $C_p[0, 1]$ onto itself that is
not open.
Recall also that some natural restrictions  on spaces $E$ in  Problem \ref{problem1} have been already noticed in \cite{LLP} and \cite{Kakol-Leiderman2}:

(i) Let $X$ be a Tychonoff space and $E$ be an infinite-dimensional  normed space. Then there exists no  sequentially continuous linear surjection from $C_p(X)$ onto $E_w$.
 (ii) If $E$ is an infinite-dimensional  metrizable lcs  and $T: C_p(X)\rightarrow E_w$ a sequentially continuous linear surjection, then the completion of $E$ is isomorphic to $\mathbb{R}^\mathbb{N}$. Here $E_w$ means $E$ endowed  with its weak topology. (iii) For cases $E=C_p(Y)$ we refer also \cite{LLP}.

A more specific version of  Problem \ref{problem1} might be also the following
\begin{problem}\label{problem3}
Let $X$ be an infinite compact space. Is it true that each semi-embedding $T:C_p(X)\rightarrow C_p(Y)$ is an embedding for any  compact space $Y$ if and only if $X$ is scattered.
\end{problem}
Note that if $T:C_p(X)\rightarrow C_p(Y)$ is  embedding  with  both $X$, $Y$  compact and $Y$ is scattered, then $X$ is scattered; this follows from \cite[Theorem III.1.2]{Ark}.  Clearly, the converse implication fails in general.
Having in mind  Problem \ref{problem3}  we prove   the following special case; this will follow from our Corollary  \ref{pro1}.
\begin{proposition}\label{pro3}
Let $X$ and  $Y$ be  compact spaces  and  $Y$  be scattered.
\begin{enumerate}
\item If $X$ is Eberlein compact, a  semi-embedding $T: C_p(X)\rightarrow C_p(Y)$ is embedding if and only if $X$ is scattered.
\item If $X$ is infinite and scattered, there exists a continuous linear injective surjection $T: C_p(X)\rightarrow C_p(X)$ which is not open. In particular,  if additionally $X$ is Eberlein compact, $T$ is not semi-embedding.
\end{enumerate}
\end{proposition}
For compact Eberlein scattered spaces $\{ [1,\alpha] \}_{\alpha<\omega_1}$, we proved in  \cite[Theorem 3.2]{Kakol-Leiderman2}  that if $\alpha\leq\beta<\omega_1$  and  $C_p([1,\alpha])$ and $C_p([1,\beta])$ are not isomorphic, then $C_p([1,\beta])$ is even not a  continuous linear image of $C_p([1,\alpha])$.

There are however compact scattered spaces $X$ which are not Eberlein  but   for which  Proposition \ref{pro3} still  holds.
\begin{corollary}\label{coro7}
If $X$ is the one-point compactification of the Isbell-Mr\'owka space and $Y$ is a compact scattered space,  then  a semi-embedding $T: C_p(X)\rightarrow C_p(Y)$ is embedding but  $X$ is not  Eberlein compact.
\end{corollary}
The above results may suggest the following natural question whether the space $(c_0)_p$ contains a closed infinite-dimensional subspace which is not isomorphic to $(c_0)_p$. This will follow  from  the next  theorem.
\begin{theorem}\label{co}
Let $ E \subset c_{0} $ be a closed infinite-dimensional subspace of the Banach space $c_0$.  Let $ \varepsilon > 0 $. Then there exists an isomorphism $ T \in \mathcal{L}(c_{0}) $ with $ \Vert T - I_{c_{0}} \Vert < \varepsilon $ such that $ T(E) $ is closed in the pointwise topology of $ c_{0} $.
\end{theorem}
\begin{corollary}\label{coo}
The space $(c_0)_p$ contains a closed infinite-dimensional subspace which is not isomorphic to $(c_0)_p$.
\end{corollary}
\section{Preliminaries and definitions}\label{Section2}
Recall the following two classes of metrizable dense subspaces of $\mathbb{R}^{\mathbb{N}}$: $$(c_0)_p=\{(x_{n})\in\mathbb{R}^{\mathbb{N}}: x_{n}\rightarrow 0\},\;(\ell_{\infty})_{p}=\{(x_{n})\in\mathbb{R}^{\mathbb{N}}: \sup_n|x_{n}|<\infty\}. $$
Let $(c_{00})_p=\{(x_{n})\in\mathbb{R}^{\mathbb{N}}: x_{n}=0\; \mbox{for sufficiently large}\; n\}.$

 The space   $(c_0)_p$ contains a copy of $(\ell_{\infty})_p$. In fact, let $(b_{n})\in c_0$ with $b_n\neq 0, n\in \mathbb{N}$. The map $$T: (\ell_{\infty})_p \rightarrow (c_{0})_p, (a_{n})\rightarrow (a_{n}b_{n})$$ is an isomorphism onto its range, so $(c_{0})_p$ contains a copy of $(\ell_{\infty})_p$.

 The spaces $(c_{00})_p$ and $(\ell_{q})_p$ for $0<q\leq\infty$ are
  proper $\sigma$-compact dense subspaces of $\mathbb{R}^{\mathbb{N}}$.

 A significant role of  the space $(c_0)_p$ is described by the following  results   from \cite[Theorem 3.1]{KMS} and \cite[Theorem 1]{BKS1}, respectively.
\begin{theorem} [K\c akol--Molto--\'Sliwa]
 For any infinite Tychonoff space $X$ the space $C_p(X)$ contains a copy of $(c_0)_p$. If  Tychonoff  $X$ contains an infinite compact subset, then $C_p(X)$ contains a closed copy of $(c_0)_p$.
 \end{theorem}
 \begin{theorem} [Banakh--K\c akol--\'Sliwa]
For any infinite Tychonoff space $X$ the space $C_p(X)$ contains a complemented copy of $(c_0)_p$ if and only if $C_p(X)$ admits a continuous linear surjection onto  $(c_0)_p$ if and only if $C_p(X)$ satisfies the Josefson-Nissenzweig property.
\end{theorem}
Consequently, if a compact space $X$ is scattered, then $C_p(X)$ contains a complemented copy of $(c_0)_p$.

 In \cite{KSZ} K\c akol, Sobota and Zdomskyy showed that there exist compact non-scattered spaces $X$ such that $C(X)$ contains a complemented copy of the Banach space $c_0$ but $C_p(X)$ does not contain any  complemented copy of $(c_0)_p$.

In order to present Corollary  \ref{pro1} let us recall that a  topological space $X$ is a \emph{$\Delta $-space} if for
every decreasing sequence $(D_{n})_{n}$ of subsets of $X$ with
empty intersection, there is a decreasing sequence $(V_{n})_n $ of open subsets of $X$, with empty intersection,
such that $D_{n}\subset V_{n}$ for every $n\in \mathbb{N}$, see  \cite{Kn}, \cite{KLe}.

Knight \cite{Kn} called $\Delta $\textit{-sets} all topological spaces $X$
satisfying the above definition. The original definition of a $\Delta $%
-set of the real line is due to Reed and van Douwen, see \cite{Re}.

Recall
that a set of real numbers $X$ is called a $Q$-set if each subset of $X$ is
a $G_{\delta }$ set in $X$. Note that the existence of uncountable $Q$-sets is
independent of ZFC. Every $Q$-set is a $\Delta $-set, but consistently the
converse is not true, see \cite{Kn}.

K\c{a}kol and Leiderman \cite{KLe} proved the following characterizations:
\begin{fact}\label{xx}  For a Tychonoff space $X$,  the strong dual $C_{p}(X)'_{\beta}$ of $C_p(X)$ carries the finest locally convex topology if and only if $X$ is a $\Delta$-space, \cite[Theorem 2.1]{KLe}.   Every compact $\Delta $-space is scattered, \cite[Theorem 3.4]{KLe}. The compact scattered space $[0,\omega _{1}]$  is not a $\Delta $-space,  but  an Eberlein compact space $X$ is a $\Delta$-space if and only if $X$ is scattered, see  \cite[Theorem 3.2]{KLe} and \cite[Theorem 3.7]{KLe}, respectively.
\end{fact}
\section{Proofs of Theorem \ref{MAIN} and  Propositions \ref{pro3}, \ref{0}, \ref{1}}
In order to prove Theorem \ref{MAIN} first we show the following crucial
\begin{proposition}\label{0}
Let $ K $ be a scattered compact space and $ E $ be a norm-closed infinite-dimensional subspace of $ C(K) $. Then $ (E, \tau_{p}) $ contains an isomorphic copy of $ (c_{0})_{p} $ that is complemented in $ C_{p}(K) $.
\end{proposition}

\begin{proof}
Since $K$ is scattered, we apply \cite[Theorem 12 (5)]{KLS} to show that there exists a sequence $ u_{1}, u_{2}, \dots $ of non-zero elements of $ E $ such that for each $ x \in K $, only finitely many $ n $'s satisfy $ u_{n}(x) \neq 0 $. Let us denote
$$ F_{n} = \mathrm{span} \, \{ u_{i} : i \geq n \}, \quad n \in \mathbb{N}. $$
By \cite[Theorem 11]{LPP} of Lotz, Peck and Porta, $ \overline{F_{1}} $ contains a sequence $ h_{1}, h_{2}, \dots $ which is equivalent to the canonical basis of $ c_{0} $. We can find numbers $ \varepsilon_{n} > 0 $ such that any sequence $ f_{n} \in C(K) $ satisfying $$ \Vert f_{n} - h_{i_{n}} \Vert < \varepsilon_{n} $$ for some subsequence $ h_{i_{n}} $ is also equivalent to the canonical basis of $ c_{0} $. Since $ h_{i} $ converges weakly to $ 0 $ and $ \overline{F_{n}} $ has finite codimension in $ \overline{F_{1}} $, we have $$ \mathrm{dist}(h_{i}, F_{n}) \to 0 $$ as $ i \to \infty $. We can choose a subsequence $ h_{i_{n}} $ such that $$ \mathrm{dist}(h_{i_{n}}, F_{n}) < \varepsilon_{n}.$$  For each $ n $, we pick $ f_{n} \in F_{n} $ such that $$ \Vert f_{n} - h_{i_{n}} \Vert < \varepsilon_{n}.$$

We have found a sequence $ (f_{n})_n $ in $ E $ which is equivalent to the canonical basis of $ c_{0} $ with the property that for each $ x \in K $, only finitely many $ n $'s satisfy $ f_{n}(x) \neq 0 $ (since there is $ n_{0} $ such that $ f(x) = 0 $ for each $ f \in F_{n_{0}} $). We claim that there is a subsequence of $ (f_{n})_n $ that generates a complemented copy of $ (c_{0})_{p} $ in the space  $ C_{p}(K) $.

Let $ a > 0 $ be such that $ \Vert f_{n} \Vert \geq a $ for every $ n $. For each $ n\in\mathbb{N} $, we choose $ x_{n} \in K $ such that $ |f_{n}(x_{n})| \geq a $. Let $ n_{k} $ be indices such that $ x_{n_{k}} $ converge to some $ y $ (see \cite[Lemma 9]{LPP} of Lotz, Peck and Porta), we can moreover ignore finitely many of these indices in order to have $ f_{n_{k}}(y) = 0 $ for each $ k $.

We choose a further subsequence $ (f_{n_{k_{l}}}) $ such that:

\bigskip

(i) $ f_{n_{k_{l}}}(x_{n_{k_{j}}}) = 0 $ for $ j < l $,

\bigskip

(ii) $ |f_{n_{k_{j}}}(x_{n_{k_{l}}})| < 2^{-(j+l)} a/3 $ for $ j < l $.

\bigskip

Such a subsequence can be chosen recursively, since for each $ j $, only finitely many $ k $'s satisfy $ f_{n_{k}}(x_{n_{k_{j}}}) \neq 0 $, and at the same time, $$ f_{n_{k_{j}}}(x_{n_{k}}) \to f_{n_{k_{j}}}(y) = 0 $$ as $ k \to \infty $.
For simplicity of the notation, we denote $ g_{j} = f_{n_{k_{j}}} $ and $ y_{j} = x_{n_{k_{j}}} $, so $$ |g_{j}(y_{j})| \geq a,\,\, y_{j} \to y $$ and $ g_{j}(y) = 0 $ for each $ j $, and

\bigskip

(i) $ g_{l}(y_{j}) = 0 $ for $ j < l $,

\bigskip

(ii) $ |g_{j}(y_{l})| < 2^{-(j+l)} a/3 $ for $ j < l $.

\bigskip

We put $$ g^{*}_{1} = \frac{1}{g_{1}(y_{1})} (\delta_{y_{1}} - \delta_{y}) $$ and recursively for $ l = 1, 2, \dots $,
$$ g^{*}_{l+1} = \frac{1}{g_{l+1}(y_{l+1})} (\delta_{y_{l+1}} - \delta_{y}) - \sum_{i=1}^{l} \frac{g_{i}(y_{l+1})}{g_{l+1}(y_{l+1})} g^{*}_{i}. $$
We want to show that $ g_{j} $'s and $ g^{*}_{j} $'s form a biorthogonal system.

By induction on $ l $, we check that $ g^{*}_{l}(g_{j}) $ equals $ 1 $ for $ j = l $ and $ 0 $ for $ j \neq l $. Note first that  $$g^{*}_{1}(g_{1}) = \frac{1}{g_{1}(y_{1})} (g_{1}(y_{1}) - g_{1}(y)) = 1 $$ and, by (i), $$ g^{*}_{1}(g_{j}) = \frac{1}{g_{1}(y_{1})} (g_{j}(y_{1}) - g_{j}(y)) = 0 $$ for $ j \geq 2 $. Concerning the induction step, we can write
\begin{align*}
g^{*}_{l+1}(g_{j}) & = \frac{1}{g_{l+1}(y_{l+1})} (g_{j}(y_{l+1}) - g_{j}(y)) - \sum_{i=1}^{l} \frac{g_{i}(y_{l+1})}{g_{l+1}(y_{l+1})} g^{*}_{i}(g_{j}) \\
& = \frac{1}{g_{l+1}(y_{l+1})} g_{j}(y_{l+1}) - \sum_{i=1}^{l} \frac{g_{i}(y_{l+1})}{g_{l+1}(y_{l+1})} g^{*}_{i}(g_{j}).
\end{align*}
The sum here equals to $ 0 $ if $ j > l $ and to $$ \frac{g_{j}(y_{l+1})}{g_{l+1}(y_{l+1})} $$ if $ j \leq l $. Therefore, $$ g^{*}_{l+1}(g_{l+1}) = \frac{1}{g_{l+1}(y_{l+1})} g_{l+1}(y_{l+1}) - 0 = 1.$$  For $ j < l+1 $, we have $$ g^{*}_{l+1}(g_{j}) = \frac{1}{g_{l+1}(y_{l+1})} g_{j}(y_{l+1}) - \frac{g_{j}(y_{l+1})}{g_{l+1}(y_{l+1})} = 0.$$  For $ j > l+1 $, we have $$ g^{*}_{l+1}(g_{j}) = \frac{1}{g_{l+1}(y_{l+1})} g_{j}(y_{l+1}) - 0 = 0 $$ by (i).

Further, by induction on $ l $, we show that $ \Vert g^{*}_{l} \Vert \leq 3/a $ and
$$ \Big\Vert g^{*}_{l+1} - \frac{1}{g_{l+1}(y_{l+1})} (\delta_{y_{l+1}} - \delta_{y}) \Big\Vert \leq \frac{2^{-l}}{a}. $$
Clearly $ \Vert g^{*}_{1} \Vert \leq 2/a \leq 3/a $.

Applying  (ii), we compute
\begin{align*}
\Big\Vert g^{*}_{l+1} - \frac{1}{g_{l+1}(y_{l+1})} & (\delta_{y_{l+1}} - \delta_{y}) \Big\Vert = \Big\Vert \sum_{i=1}^{l} \frac{g_{i}(y_{l+1})}{g_{l+1}(y_{l+1})} g^{*}_{i} \Big\Vert \\
 & \leq \frac{1}{a} \cdot \frac{3}{a} \sum_{i=1}^{l} |g_{i}(y_{l+1})| \leq \frac{1}{a} \cdot \frac{3}{a} \sum_{i=1}^{l} \frac{2^{-(i+l)} a}{3} \leq \frac{2^{-l}}{a}.
\end{align*}
Hence we have  $$ \Vert g^{*}_{l+1} \Vert \leq \frac{2}{a} + \frac{2^{-l}}{a} \leq \frac{3}{a}. $$

It follows that $ g^{*}_{l}(f) \to 0 $ for any $ f \in C(K) $, because
\begin{align*}
|g^{*}_{l+1}(f)| & \leq \frac{2^{-l}}{a} \Vert f \Vert + \Big| \frac{1}{g_{l+1}(y_{l+1})} \big( \delta_{y_{l+1}}(f) - \delta_{y}(f) \big) \Big| \\
 & \leq \frac{2^{-l}}{a} \Vert f \Vert + \frac{1}{a}|f(y_{l+1})-f(y)|,
\end{align*}
which tends to $ 0 $.

Finally, let $ S $ be the operator defined in the following form  $$ f \mapsto (g^{*}_{l}(f))_l.$$  Then $ S $ is a continuous linear mapping from $ C_{p}(K) $ to the space $ (c_{0})_{p} $. This is because each $ g^{*}_{l} $ is a linear combination of a finite number of Dirac measures.

Let $ T $ be the operator defined as follows $$ (z_{i})_i \mapsto \sum_{i=1}^{\infty} z_{i}g_{i}.$$  Then $ T $ is a continuous linear mapping from $ (c_{0})_{p} $ to the space $ C_{p}(K) $. Indeed, we know that $ T $ is a norm-norm embedding of $ c_{0} $ into $ C(K) $, the pointwise-pointwise continuity follows from the property that for each $ x \in K $, only finitely many $ n $'s satisfy $ f_{n}(x) \neq 0 $.

Next, we realize that $ ST $ is the identity on $ c_{0} $, since for $ n \in \mathbb{N} $, we have $$ ST(e_{n}) = S(g_{n}) = (g^{*}_{l}(g_{n}))_l = e_{n}.$$  It follows that $ T(c_{0}) $ is a complemented copy of $(c_{0})_{p} $ in $ C_{p}(K) $, which is witnessed by the projection $ P = TS$.  Indeed, $$ P^{2} = TSTS = TidS = TS = P,$$   $$ P(C(K)) = TS(C(K)) \subset T(c_{0}) $$ and $ P(T(z)) = TST(z) = T(z) $ for each $ z \in c_{0} $.
\end{proof}
Next we show  Proposition \ref{1}  which will be used to prove  Theorem \ref{MAIN}.  Recall first that a lcs $E$ admits a fundamental sequence of bounded sets if $E$ has a sequence $(S_n)_n$ of bounded sets such that every bounded set in $E$ is contained in some $S_{m}$.
\begin{proposition} \label{1} Let $E$ be an infinite-dimensional separable reflexive Banach space. Then for every non-scattered compact space $X$, the space $C_p(X)$ contains a closed subspace $F$ isomorphic to the space $E_w$ endowed with the weak topology.
Clearly, $F$ is $\sigma$-compact and has a fundamental sequence of bounded sets. \end{proposition}

\begin{proof} The compact space $X$ is non-scattered, so there exists a continuous surjection $\phi: X \to [0,1].$ Clearly, the compact metrizable space $Y=(B_{E^*}, w^*)$ is locally connected in each point, so it is a Peano continuum. Thus, by the Hahn-Mazurkiewicz Theorem \cite[Theorem 2]{Ward} there exists a continuous surjection $\psi: [0,1] \to Y.$ The continuous surjection $\psi \circ \phi: X \to Y$ is a quotient map, so $C_p(X)$ contains a closed copy of $C_p(Y)$, see \cite[Corollary 0.4.8]{Ark}. It is easy check that the linear map $S:E_w \to C_p(Y),\,\,\, x \to f_x, $ where $$f_x: Y \to \mathbb{R},\,\,\,  g \to g(x),$$ is an isomorphism onto its range.

We shall prove that $S(E_w)$ is closed in $C_p(Y)$. Let $$(x_{\alpha}) \subset E,\,\,\, h\in C_p(Y)$$ and $$Sx_{\alpha} \to_{\alpha} h$$ in $C_p(Y)$. Then $g(x_{\alpha}) \to_{\alpha} h(g)$ for every $g\in B_{E^*}.$ It follows that the map $$\hat{h}: E^* \to \mathbb{R},\,\,\, g \to \lim_{\alpha} g(x_{\alpha})$$ is well defined and linear; clearly, $\hat{h}(g)=h(g)$ for $g\in B_{E^*}$, so $\hat{h}$ is continuous on $Y$. Hence $\hat{h}$ is continuous on $(B_{E^*}, \| \cdot \|),$ so it is continuous on the set $(E^*, \|\cdot\|)$. Thus $\hat{h}\in E^{**}$. By reflexivity of $E$, there exists $x_0 \in E$ such that $h(g)=\hat{h}(g)=g(x_0)$ for every $g\in B_{E^*}$. Thus $$h=f_{x_0}=S(x_0)\in S(E_w).$$ It follows that $S(E_w)$ is closed in $C_p(Y)$.
We have shown that $C_p(X)$ contains a closed copy of $E_w$. \end{proof}

\begin{corollary}\label{2}  For every non-scattered compact space $X$, the space $C_p(X)$ contains an infinite-dimensional closed subspace $F$ without any copy of $(c_{00})_p$. \end{corollary}
\begin{proof} Note that  the metrizable  lcs $(c_{00})_p$ contains no fundamental sequence of bounded sets, otherwise would be normed by the  Kolmogoroff theorem, see \cite[Proposition 6.9.4]{Jarchow},  which is impossible.  A  direct argument:  For every $(\alpha_n) \in \mathbb{R}^{\mathbb{N}}$ the set $\{\alpha_ne_n: n\in \mathbb{N}\}$ is bounded in $\mathbb{R}^{\mathbb{N}}.$ For each sequence $(V_n)$ of bounded sets in $(c_{00})_p$ there exists $(\beta_n) \in \mathbb{R}^{\mathbb{N}}$ such that $\beta_ne_n \not \in V_n$ for every $n\in \mathbb{N}.$ Then $\{\beta_ne_n: n\in \mathbb{N}\}\not \subset V_m$ for any $m\in \mathbb{N}.$  Thus $(c_{00})_p$ contains no fundamental sequence of bounded sets. Using  last Proposition \ref{1}  we complete the proof. \end{proof}
\begin{proof}[Proof of Theorem \ref{MAIN}] Using both Proposition \ref{0} and  Corollary \ref{2} we  complete the proof of Theorem \ref{MAIN}.
\end{proof}

Using Theorem \ref{Main2}, Proposition \ref{1} and Corollary \ref{2} and their proofs we get the following

\begin{theorem} \label{3} For a compact space $X$ the following  are equivalent:

(1) $X$ is scattered;

(2) $C_p(X)$ contains no infinite-dimensional subspace with a fundamental sequence of bounded sets;

(3) $C_p(X)$ contains no copy of $E_w$ for any infinite-dimensional separable Banach space $E$;

(4) Every infinite-dimensional closed subspace of $C_p(X)$ contains a copy of $(c_0)_p.$ \end{theorem}
We provide a short proof of the equivalence (1) $\Leftrightarrow$ (3).  If $X$ is scattered,  $C_p(X)$ is Fr\'echet-Urysohn \cite[Theorem III.1.2]{Ark}.   Recall a topological space $W$ is  Fr\'echet-Urysohn if for each  $A\subset W$ and  $x\in\overline{W}$ there exists a sequence in $A$ converging to $x$. Assume  $E$ is an infinite-dimensional  Banach space and $E_w$ is embedded into $C_p(X)$. Then $E_w$ is Fr\'echet-Urysohn, as well. Since Fr\'echet-Urysohn lcs are bornological  \cite[Lemma 14.4.3]{KKPS}, the space $E_w$ is bornological (i.e. every absolutely convex and bornivorous subset of $E_w$ is a neighbourhood of zero), what implies that the norm topology of $E$  and the weak topology of $E_w$ coincide. Hence $E$  is finite-dimensional, a contradiction. The converse (3) $\Rightarrow$ (1) follows from Proposition \ref{1}.
 \begin{corollary}\label{coro4}
If X is an infinite compact scattered space and $C_p(Y)$ is an infinite-dimensional closed subspace of $C_p(X)$ for some Tychonoff  space $Y$, then $Y$ is compact and scattered and $C_p(Y)$ contains a copy of $(c_0)_p$ complemented in $C_p(X)$.
\end{corollary}
\begin{proof}
The space $Y$ is compact by \cite[Theorem 3.11]{KL}, and then  $Y$ is scattered by  \cite[Theorem 3.1.2]{Ark}, and we apply Theorem \ref{MAIN}.
\end{proof}
Recall that a lcs  $E$ is \emph{quasibarrelled}, if every bornivorous absolutely convex closed set in  $E$ is a neighbourhood of zero. Recall also that a  lcs $E$ is called \emph{strongly distinguished} \cite{KS} if the strong dual $E'_{\beta}$ of $E$ carries the finest locally convex topology.

In order to prove Proposition \ref{pro3} we need the following Proposition \ref{pro5} which uses some ideas  contained in \cite[Theorem 1.2]{Kakol-Leiderman2}.
\begin{proposition} \label{pro5} Any polar linear map $T$ from a strongly distinguished lcs $E$ onto a quasibarrelled lcs $F$ is open.\end{proposition}

\begin{proof} Let $(U_s)_{s\in S}$ be a base of neighbourhoods of zero in $E$ such that $V_s=T(U_s)$ is closed in $F$ for every $s\in S.$ We shall prove that the set $V_s$ is bornivorous in $F$ for every $s\in S.$ Let $s\in S.$

The adjoint map $T^{*}: F'_{\beta}\rightarrow E'_{\beta}$ is injective and open onto its range $T^*(F'_{\beta})$, since $E'_{\beta}$ carries the finest locally convex topology and the finest locally topology is inherited by vector subspaces.

Let $A$ be a bounded subset of $F$. Then the polar $A^\circ$ of $A$ is a neighbourhood of zero in $F'_{\beta}$, so  $T^*(A^\circ)$ is open in $T^*(F'_{\beta})$.

 Thus there exists an absolutely convex bounded subset $B$ of $E$ such that $B^\circ \cap T^*(F'_{\beta}) \subset T^*(A^\circ).$ Then we have $$[T(B)]^{\circ}\subset (T^{*})^{-1}(B^{\circ})=(T^{*})^{-1}[B^{\circ}\cap T^{*}(F'_{\beta})]\subset A^{\circ}.$$ By the bipolar theorem we get $$A\subset A^{\circ\circ}\subset T(B)^{\circ\circ}\subset\overline{T(B)}.$$ The set $B$ is bounded in $E$, so $B \subset tU_s$ for some $t\in \mathbb{R}$. $\\$ Then $T(B) \subset tT(U_s)=tV_s,$ so $A\subset \overline{T(B)} \subset \overline{tV_s} = tV_s.$ $\\$ It follows that $V_s$ is bornivorous in $F$ for every $s\in S.$

Let $s\in S.$ The set $U_s$ contains an absolutely convex neighbourhood of zero $W_s$  in $E$ and $W_s$ contains $U_r$ for some $r\in S$. Then $$V_r=T(U_r)\subset T(W_s) \subset \overline{T(W_s)}\subset \overline{T(U_s)}=\overline{V_s}=V_s,$$ so $\overline{T(W_s)}$ is a bornivorous absolutely convex closed set in $F$. Thus $\overline{T(W_s)}$ is a neighbourhood of zero in $F$, so $V_s$ is a neighbourhood of zero in $F$. It follows that the map $T$ is open. \end{proof}
For any $\Delta$-space $X$, the space $C_p(X)$ is strongly distinguished (by \cite[Theorem 6]{KS}) and quasibarrelled (by \cite[Corollary 11.7.3]{Jarchow}); any quotient of a quasibarrelled lcs is quasibarrelled, so we get the following.
\begin{corollary}\label{pro1}
Let $X$ be a  $\Delta$-space. A polar map $T$ from $C_p(X)$ onto an lcs $F$ is open if and only if $F$ is quasibarrelled.
\end{corollary}
 The unit segment $[0,1]$, as well as $[0,\omega_1]$, are not $\Delta$-spaces (see Fact \ref{xx}), and a challenging problem of whether
 a non-open continuous linear polar surjection $T : C_p[0,1] \rightarrow C_p[0,1]$ exists, remains unsolved, see  \cite{Kakol-Leiderman2} and \cite{LLP}.
 Note (what will be used below) that if $E$ is a Fr\'echet-Urysohn lcs, every subspace of $E$ enjoys the same property, and every Fr\'echet-Urysohn lcs is bornological, hence quasibarrelled,
 see \cite[Lemma 14.4.3]{KKPS}.
\begin{proof}[Proof of Proposition \ref{pro3}] (1) If $Y$ is a compact scattered space, then $C_p(Y)$ is a Fr\'echet-Urysohn space, see \cite[Theorem III. 1.2]{Ark}. As the Fr\'echet-Urysohn property is inherited by subspaces, the image $T(C_p(X))$  is also Fr\'echet-Urysohn, so it  must be quasibarrelled by \cite[Lemma 14.4.3]{KKPS}. Assume that $T$ is an embedding. Then $C_p(X)$ and $T(C_p(X))$ are isomorphic. Hence  $C_p(X)$ is also Fr\'echet-Urysohn. This implies that $X$ is scattered. For the converse, assume that $X$ is scattered. Then $X$ is a $\Delta$-space, see  Fact \ref{xx}. Assume  $T$ is semi-embedding.   Finally we apply Corollary \ref{pro1} to get that $T$ is embedding.

(2) Since $X$ is scattered, $X$ is zero-dimensional, and there exists in $X$ an infinite sequence $x_n \rightarrow x$. Set $Y=\{x_n:n\in\mathbb{N}\}\cup\{x\}$.  It is known that then  $Y$ is a retraction of $X$, i.e. there exists a continuous map $r: X\rightarrow Y$ such that $r(y)=y$ for all $y\in Y$. Indeed, we embed $X$ into some Cantor cube $D^W$. Hence we look at  $Y$ as a subspace of $X$, and the latter space is a subspace of $D^W$. By \cite[Theorem 2]{Haydon} the space $D^W$ is $AE(0)$-space, so there exists a continuous retraction   $r: D^W\rightarrow Y$. The restriction map $r|X$ is a retraction from $X$ onto $Y$.
The function $q: Y \to X$ with $q(x)=x$ and $q(x_n)=x_{n+1}$ for every $n\geq 1$ is continuous. The map $h=q\circ r: X \to X$ is continuous and $\{h^k(x_1): k\geq 0\}=\{x_k: k\geq 1\}$, so the orbit of $h$ at some point is infinite. Let $\lambda \in {\mathbb R}$ with $|\lambda|>1.$

By \cite[Proposition 2.1]{LLP}, the map $T: C_p(X)\rightarrow C_p(X), Tf=\lambda f+f\circ h$ is linear, continuous, non-open and surjective. Note that the map $T$ is injective. Indeed, let $f\in C(X)$ with $Tf=0$. Then $f(h(z))=-\lambda f(z)$ for every $z\in X.$ Hence, by induction we get $$f(h^k(z))=(-\lambda)^kf(z)$$ for every $k\geq 0, z\in X.$ Thus $\sup_{k\geq 0} |\lambda|^k|f(z)|\leq \|f\|_{\infty}<\infty,$ so $f(z)=0$ for every $z\in X.$ Hence $f=0$, so $T$ is injective.  Now assume that $X$ is Eberlein compact. Then by part (1) we deduce that $T$ is not semi-embedding.  The proof of part (2) is completed.
\end{proof}
\begin{proof}[Proof of Corollary \ref{coro7}]
By \cite[Theorem 3.10, Example 3.17]{KLe} the space $X$ is a $\Delta$-space which is  scattered separable, and Corollary \ref{pro1} applies. Note that $X$ is not Eberlein \cite[p.16]{Ark}, since separable Eberlein compact spaces are metrizable \cite[Theorem III.3.6]{Ark}, but the Isbell-Mr\'owka space is not metrizable.
\end{proof}
\section{Proof of Theorem \ref{MAIN3} and Corollary \ref{MAIN2}}
\begin{proof}[Proof of Theorem \ref{MAIN3}]
 $(2) \Rightarrow (1)$ Let $T$ be an isomorphism from $(c_{00}(\Gamma))_p$ to $C_p(X)$. Let $f_{\gamma}=Te_{\gamma}$ for $\gamma \in \Gamma.$ Then $(f_{\gamma}) \subset (C_p(X)\setminus \{0\}).$ For every $(t_{\gamma})\in \mathbb{R}^{\Gamma}$ we have $t_{\gamma}e_{\gamma} \to_{\gamma} 0$ in $(c_{00}(\Gamma))_p$ and $t_{\gamma}f_{\gamma} \to_{\gamma} 0$ in $C_p(X)$.

 Thus for every $x\in X$ and every $(t_{\gamma})\in \mathbb{R}^{\Gamma}$ we deduce  $t_{\gamma}f_{\gamma}(x) \to_{\gamma} 0$. It follows that $(f_{\gamma}(x)) \in c_{00}(\Gamma)$ for every $x\in X.$

Thus the family $$\{U_{\gamma}: \gamma \in \Gamma\}:= \{f_{\gamma}^{-1}(\mathbb{R} \setminus \{0\}): \gamma \in \Gamma\}$$ of non-empty open subsets of $X$ is point-finite, so for every $W\subset X$ the set $\Gamma_W:=\{\gamma \in \Gamma: U_{\gamma}=W\}$ is finite. Since $\bigcup \{\Gamma_W: W\subset X\}=\Gamma$,  the set $${\mathcal W}:= \{W\subset X: \Gamma_W \neq \emptyset\}$$ is uncountable and $|{\mathcal W}|=|\Gamma|$. Clearly, $\Gamma_W \cap \Gamma_V=\emptyset,$ if $W\neq V.$ Let $\gamma(W) \in \Gamma_W$ for $W\in {\mathcal W}$. Then the family $\{U_{\gamma(W)}: W \in {\mathcal W}\}$ is uncountable, so the point-finite family ${\mathcal U}=\{U_{\gamma}: \gamma \in \Gamma\}$ is uncountable and $|{\mathcal U}|=|\Gamma|$. By  \cite[Lemma 4.2]{Rosenthal}, the space $X$ does not satisfy the countable chain condition, so there exists an uncountable family of pairwise disjoint non-empty open subsets of $X$.

$(1) \Rightarrow (3)$ Let $\{V_{\gamma}: \gamma \in \Gamma\}$ be a family of pairwise disjoint non-empty open subsets of $X$ and let $x_{\gamma}\in V_{\gamma}$ for $\gamma \in \Gamma.$ For every $\gamma \in \Gamma$ there exists a continuous function $f_{\gamma}: X \to [0,1]$ such that $f_{\gamma}(x_{\gamma})=1$ and $f_{\gamma}(x)=0$ for each $x\in (X\setminus V_{\gamma})$.

Let $t=(t_{\gamma}) \in c_0(\Gamma)$. The function $g_t: X\to \mathbb{R}, g_t(x)=\sum_{\gamma \in \Gamma} t_{\gamma} f_{\gamma}(x)$ is well defined. We shall prove that $g_t$ continuous.

Let $\varepsilon >0$. The set $M_{\varepsilon}=\{\gamma \in \Gamma: |t_{\gamma}|\geq \varepsilon\}$ is finite and for $x\in (X\setminus\bigcup_{\gamma \in \Gamma} V_{\gamma})$ we have $g_t(x)=0.$ Thus we have $$\{x\in X: |g_t(x)|\geq \varepsilon \}= \bigcup_{\gamma \in \Gamma} \{x\in V_{\gamma}: |g_t(x)|\geq \varepsilon\}=$$ $$\bigcup_{\gamma \in \Gamma} \{x\in V_{\gamma}: |t_{\gamma}| f_{\gamma}(x)\geq \varepsilon\}=\bigcup_{\gamma\in M_{\varepsilon}} \{x\in X: f_{\gamma}(x)\geq \varepsilon/|t_{\gamma}|\},$$ so the set
$\{x\in X: |g_t(x)|\geq \varepsilon \}$ is closed.

Hence the set $g_t^{-1}((-\varepsilon, \varepsilon))=\{x\in X: |g_t(x)|< \varepsilon \}$ is open.

Let $s\in (\mathbb{R}\setminus \{0\})$ and $\varepsilon \in (0, |s|).$ Then $0\not \in (s-\varepsilon, s+\varepsilon)$, so $$g_t^{-1}((s-\varepsilon, s+\varepsilon))= \bigcup_{\gamma \in \Gamma} \{x\in V_{\gamma}: g_t(x)\in (s-\varepsilon, s+\varepsilon)\}=$$ $$\bigcup_{\gamma \in \Gamma} \{x\in V_{\gamma}: t_{\gamma}f_{\gamma}(x)\in (s-\varepsilon, s+\varepsilon)\}=\bigcup_{\gamma \in \Gamma} (t_{\gamma}f_{\gamma})^{-1} ((s-\varepsilon, s+\varepsilon)).$$ Thus the set $g_t^{-1}((s-\varepsilon, s+\varepsilon))$ is open.

It follows that the function $g_t$ is continuous.

Clearly, $F:=\{g_t: t\in c_0(\Gamma) \}$ is a linear subspace of $C_p(X)$. We shall prove that $F$ is isomorphic to $(c_0(\Gamma))_p$. The linear map $$T: (c_0(\Gamma))_p \to F, t=(t_{\gamma}) \to g_t=\sum_{\gamma \in \Gamma} t_{\gamma} f_{\gamma}$$ is continuous. Indeed, let $k\in \mathbb{N}, y_1, \ldots, y_k \in X, \varepsilon >0$ and $$U=\{ f\in F: |f(y_i)| < \varepsilon \; \mbox{for}\; 1 \leq i \leq k \}.$$ Let $m\in \mathbb{N}$ with $m> \varepsilon^{-1}$. For some finite set $A\subset \Gamma$ we have $$\{ y_1, \ldots, y_k\} \subset \bigcup_{\gamma \in A}V_{\gamma} \cup \Big( X\setminus \bigcup_{\gamma \in \Gamma} V_{\gamma} \Big).$$ Clearly, the sets $$W_{B, l}=\{ (t_{\gamma}) \in c_0(\Gamma): \max_{\gamma \in B} |t_{\gamma}| <l^{-1} \},$$ where $B$ is a finite subset of $\Gamma$ and $l\in \mathbb{N}$ form a base of neighbourhoods of zero in $(c_0(\Gamma))_p$.

We have $$T(W_{A,m})= \Big\{ \sum_{\gamma \in \Gamma} t_{\gamma} f_{\gamma}: (t_{\gamma}) \in c_0(\Gamma), \max_{\gamma \in A} |t_{\gamma}| <m^{-1} \Big\} \subset U.$$ Indeed, let $t\in W_{A, m}$ and $1 \leq i \leq k.$ If $y_i \in \bigcup_{\gamma \in A} V_{\gamma}$, then $y_i \in V_{\gamma}$ for some $\gamma \in A$, so $|g_t(y_i)|=|t_{\gamma}|f_{\gamma}(y_i)\leq |t_{\gamma}| < m^{-1} < \varepsilon$; if $y_i \in (X\setminus \bigcup_{\gamma \in \Gamma} V_{\gamma})$, then $|g_t(y_i)|=0< \varepsilon$. Thus $T(t)=g_t\in U.$ Hence $T(W_{A, m})\subset U.$
It follows that $T$ is continuous.

Let $l\in \mathbb{N}$ and let $B$ be a finite subset of $\Gamma$. Let $\varepsilon \in (0, l^{-1}).$ Let $$W=\{ f\in F: |f(x_{\gamma})| < \varepsilon\; \mbox{for}\; \gamma \in B\}.$$ Let $f\in W.$ Then $f=g_t$ for some $t=(t_{\gamma}) \in c_0(\Gamma)$ with $|t_{\gamma}| < \varepsilon <l^{-1}$ for $\gamma \in B$; so $f\in T(W_{B,l})$. Thus $W\subset T(W_{B, l});$ so $T$ is open.

It follows that $F$ is isomorphic to $(c_0(\Gamma))_p$.

\smallskip
 Now we shall prove that the subspace $F$ of $C_p(X)$ is closed.

Let $(h_s)_{s\in S} \subset F, h \in C_p(X)$ and $h_s \to_s h$ in $C_p(X)$. Let $s\in S$. Then  $$h_s=g_{t_s}=\sum_{\gamma \in \Gamma} t_{s, \gamma} f_{\gamma}$$ for some $t_s=(t_{s, \gamma})\in c_0(\Gamma)$. Put $\beta_{\gamma}=h(x_{\gamma}), \gamma \in \Gamma.$

Let $\gamma \in \Gamma.$ Then $$t_{s, \gamma} =t_{s, \gamma}f_{\gamma}(x_{\gamma})=h_s(x_{\gamma}) \to_s h(x_{\gamma})=\beta_{\gamma},$$ so $h_s(x)= t_{s, \gamma}f_{\gamma}(x) \to_s \beta_{\gamma}f_{\gamma}(x)$ for every $x\in V_{\gamma}.$ Thus $h(x)=\beta_{\gamma}f_{\gamma}(x)$ for every $x\in V_{\gamma}, \gamma \in \Gamma.$ Hence $$h(x)=\sum_{\gamma \in \Gamma} \beta_{\gamma}f_{\gamma}(x)$$ for $x\in X.$

We prove that $(\beta_{\gamma})\in c_0(\Gamma).$ Let $\varepsilon>0.$ $\\$ Put $$P_{\varepsilon}=\{\gamma \in \Gamma: |\beta_{\gamma}|\geq \varepsilon\}.$$ The function $h$ is continuous, so the set $$D_{\varepsilon}=\{x\in X: |h(x)|\geq \varepsilon \}$$ is closed, hence compact. We have $$D_{\varepsilon}=\bigcup_{\gamma \in \Gamma} \{x\in V_{\gamma}: |h(x)|\geq \varepsilon\}=\bigcup_{\gamma \in \Gamma} \{x\in V_{\gamma}: |\beta_{\gamma}|f_{\gamma}(x)\geq \varepsilon\}=$$ $$\bigcup_{\gamma\in P_{\varepsilon}} V_{\gamma} \cap \{x\in X: |\beta_{\gamma}|f_{\gamma}(x)\geq \varepsilon\} \subset \bigcup_{\gamma\in P_{\varepsilon}} V_{\gamma}.$$ Hence $$\{D_{\varepsilon}\cap V_{\gamma}: \gamma\in P_{\varepsilon}\}$$ is an open cover of the compact set $D_{\varepsilon}.$ The sets $$D_{\varepsilon}\cap V_{\gamma}, \gamma\in P_{\varepsilon}$$ are pairwise disjoint and non-empty, since $$x_{\gamma}\in D_{\varepsilon}\cap V_{\gamma}, \gamma\in P_{\varepsilon}.$$

Thus $P_{\varepsilon}$ is finite for every $\varepsilon>0.$ It follows that $(\beta_{\gamma})\in c_0(\Gamma),$ so $h\in F$.
Thus the subspace $F$ is closed in $C_p(X).$

$(3) \Rightarrow(2)$ is obvious. $(1) \Leftrightarrow (4)$ follows from \cite[Theorem 4.5]{Rosenthal}.

$(1) \Rightarrow(5)$: By \cite[Theorem 4.5]{Rosenthal} the Banach space $C(X)$ contains a non-separable compact set $D$  in the space $C(X)_w$. Since the pointwise topology of $C(X)$ is weaker than the weak topology of $C(X)$ and both topologies coincide on $D$,  the implication follows.

$(5) \Rightarrow(1)$: Assume that $D$ is a compact and non-separable subset of $C_p(X)$. Let $B$ be the closed unit ball in $C(X)$ (which is closed in $C_p(X)$).  Then there exists $m\in\mathbb{N}$ such that $D_m=D\cap mB$ is still non-separable in $C_p(X)$. Clearly $D_m$ is  compact in $C_p(X)$ and  bounded in $C(X)$, so Grothendieck's theorem \cite[Theorem 4.2]{Floret} applies  to derive that $D_m$ is weakly compact in $C(X)$. Again \cite[Theorem 4.5]{Rosenthal} applies to get the item (1).
\end{proof}

\begin{proof}[Proof of Corollary \ref{MAIN2}]
 If $X$ is non-separable, then the family $\{\{x\}: x \; \mbox{is isolated point of}\; X\}$ of pairwise disjoint open subsets of $X$ is uncountable. \end{proof}

A slight modification of the proof of Theorem \ref{MAIN3} and \cite[Remark, p. 227]{Rosenthal} shows also the following
\begin{theorem} Let $X$ be a compact space. Then the space $C_p(X)$ contains a closed copy of $(c_{0}(\Gamma))_p$ for some set $\Gamma$ if and only if there is a family $\{V_{\gamma}: \gamma \in \Gamma\}$ of pairwise disjoint non-empty open subsets of $X$.
\end{theorem}

\section{Proof of Theorem \ref{co} and Corollary \ref{coo}}

\begin{proof}[Proof of Theorem \ref{co}]

We can suppose that $ \varepsilon < 1 $. Let us assume first that $ E $ has infinite codimension. Let $$ A_{0} \subset A_{1} \subset A_{2} \subset \dots $$ be subspaces of $ c_{0}^{*} $ such that $ A_{n} $ has dimension $ n $ and $ \bigcup_{n=1}^{\infty} A_{n} $ is dense in the annihilator $$ E^{\bot} = \{ x^{*} \in c_{0}^{*} : (\forall x \in E)(x^{*}(x) = 0) \}.$$  For each $ n \in \mathbb{N} $, we choose $ a_{n} \in c_{0} $ such that $ u^{*}(a_{n}) = 0 $ for $ u^{*} \in A_{n-1} $ and $ u^{*}(a_{n}) \neq 0 $ for some (equivalently any) $ u^{*} \in A_{n} \setminus A_{n-1} $. Also, for each $ n \in \mathbb{N} $, we choose $ u^{*}_{n} \in A_{n} $ with $ u^{*}_{n} \neq 0 $ such that $ u^{*}_{n}(a_{m}) = 0 $ for $ 1 \leq m \leq n - 1 $ (the intersection of $ n - 1 $ hyperplanes in an $ n $-dimensional space has dimension at least $ 1 $). Necessarily, $ u^{*}_{n} \notin A_{n-1} $, since once $ u^{*}_{n} \in A_{n-1} $, then $ u^{*}_{n} \in A_{n-2} $ (as $ u^{*}_{n}(a_{n-1}) = 0 $), $ u^{*}_{n} \in A_{n-3} $ (as $ u^{*}_{n}(a_{n-2}) = 0 $), etc., and finally $ u^{*}_{n} \in A_{0} = \{ 0\} $. It follows that $ u^{*}_{n}(a_{n}) \neq 0 $.

So, we have found a sequence $ a_{n} $ in $ c_{0} $ and a sequence $ u^{*}_{n} $ in $ c_{0}^{*} $ such that $ u^{*}_{n}(a_{m}) \neq 0 $ if and only if  $ n = m $, and
$$ E = \Big( \bigcup_{n=1}^{\infty} A_{n} \Big)_{\bot} = \bigcap_{n=1}^{\infty} \{ x \in c_{0} : u^{*}_{n}(x) = 0 \}. $$ Let us choose numbers $ \varepsilon_{n} > 0 $ such that
$$ \sum_{n=1}^{\infty} \varepsilon_{n} \leq \frac{1}{2} \varepsilon \quad \textrm{and} \quad \prod_{n=1}^{\infty} (1+\varepsilon_{n}) \leq 2. $$
Next, we construct recursively for $ n = 1, 2, \dots $:
\begin{itemize}
\item $ S_{0} = I_{c_{0}} $,

\bigskip

\item $ b_{n} = S_{n-1}(a_{n}) $,

\bigskip

\item $ v^{*}_{n} $ continuous with respect to the pointwise topology on $ c_{0} $ chosen such that $ \Vert v^{*}_{n} - u^{*}_{n} \Vert < \frac{\varepsilon_{n}}{\Vert b_{n} \Vert} |v^{*}_{n}(b_{n})| $ (we explain later that such choice is possible),

\bigskip

\item $ T_{n}(x) = x + (u^{*}_{n}(x) - v^{*}_{n}(x)) \frac{1}{v^{*}_{n}(b_{n})} b_{n} $ for $ x \in c_{0} $,

\bigskip

\item $ S_{n} = T_{n} \circ T_{n-1} \circ \dots \circ T_{1} $.
\end{itemize}
We prove by induction that
\begin{itemize}
\item[{(i)}] $ u^{*}_{n}(b_{n}) = u^{*}_{n}(a_{n}) \neq 0 $, $ v^{*}_{m}(b_{n}) = 0 $ for $ 1 \leq m \leq n - 1 $ and $ u^{*}_{m}(b_{n}) = 0 $ for $ m \geq n + 1 $,

\bigskip

\item[{(ii)}] the choice of $ v^{*}_{n} $ is possible,

\bigskip

\item[{(iii)}] $ \Vert T_{n} - I_{c_{0}} \Vert \leq \varepsilon_{n} $, and so $ \Vert T_{n} \Vert \leq 1 + \varepsilon_{n} $,

\bigskip

\item[{(iv)}] $ \Vert S_{n} \Vert \leq 2 $ and $ \Vert S_{n} - S_{n-1} \Vert \leq 2\varepsilon_{n} $,

\bigskip

\item[{(v)}] $ v^{*}_{n}(T_{n}(x)) = u^{*}_{n}(x) $, $ v^{*}_{m}(T_{n}(x)) = v^{*}_{m}(x) $ for $ 1 \leq m \leq n - 1 $ and $ u^{*}_{m}(T_{n}(x)) = u^{*}_{m}(x) $ for $ m \geq n + 1 $,

\bigskip

\item[{(vi)}] $ v^{*}_{m}(S_{n}(x)) = u^{*}_{m}(x) $ for $ 1 \leq m \leq n $ and $ u^{*}_{m}(S_{n}(x)) = u^{*}_{m}(x) $ for $ m \geq n + 1 $.

\bigskip

\end{itemize}
Concerning the first step $ n = 1 $, we just note that (i) holds since $ b_{1} = a_{1} $. The other properties can be proved in the same way as in the induction step $ n - 1 \to n $. Let us assume that $ n \in \mathbb{N} $ and that the properties are shown for $ 1 \leq m \leq n - 1 $. We show that they hold also for $ n $.

(i) We use property (vi) for $ n - 1 $ as follows. For $ 1 \leq m \leq n - 1 $, we compute $$ v^{*}_{m}(b_{n}) = v^{*}_{m}(S_{n-1}(a_{n})) = u^{*}_{m}(a_{n}) = 0.$$  For $ m \geq n $, we compute $$ u^{*}_{m}(b_{n}) = u^{*}_{m}(S_{n-1}(a_{n})) = u^{*}_{m}(a_{n}),$$ that is non-zero for $ m = n $ and zero for $ m \geq n + 1 $.

(ii) Once we know that $ u^{*}_{n}(b_{n}) \neq 0 $, it is sufficient to note that pointwise continuous functionals are norm-dense in $ c_{0}^{*} $ and $$ \Big\{ v^{*} \in c_{0}^{*} : \Vert v^{*} - u^{*}_{n} \Vert < \frac{\varepsilon_{n}}{\Vert b_{n} \Vert} |v^{*}(b_{n})| \Big\} $$ is an open set containing $ u^{*}_{n} $.

(iii) We have $$\Vert (T_{n} - I_{c_{0}})(x) \Vert = \Big\Vert (u^{*}_{n}(x) - v^{*}_{n}(x)) \frac{1}{v^{*}_{n}(b_{n})} b_{n} \Big\Vert $$  $$ \leq \Vert u^{*}_{n} - v^{*}_{n} \Vert \Vert x \Vert \frac{\Vert b_{n} \Vert}{|v^{*}_{n}(b_{n})|} \leq \varepsilon_{n} \Vert x \Vert $$ for each $ x \in c_{0}.$

(iv) Using (iii), we can compute $$ \Vert S_{n} \Vert \leq (1+\varepsilon_{n}) (1+\varepsilon_{n-1}) \dots (1+\varepsilon_{1}) \leq 2.$$  Thus we derive $$ \Vert S_{n} - S_{n-1} \Vert = \Vert (T_{n} - I_{c_{0}}) \circ S_{n-1} \Vert \leq 2\varepsilon_{n}.$$

(v) Applying  (i) we deduce that  $$ v^{*}_{n}(T_{n}(x)) = v^{*}_{n}(x) + (u^{*}_{n}(x) - v^{*}_{n}(x)) \frac{1}{v^{*}_{n}(b_{n})} v^{*}_{n}(b_{n}) =$$ $$ v^{*}_{n}(x) + (u^{*}_{n}(x) - v^{*}_{n}(x)) = u^{*}_{n}(x).$$  For  $1 \leq m \leq n - 1 $, we write $$ v^{*}_{m}(T_{n}(x)) = v^{*}_{m}(x) + (u^{*}_{n}(x) - v^{*}_{n}(x)) \frac{1}{v^{*}_{n}(b_{n})} v^{*}_{m}(b_{n}) = v^{*}_{m}(x) + 0.$$  For $ m \geq n + 1 $, we have $$ u^{*}_{m}(T_{n}(x)) = u^{*}_{m}(x) + (u^{*}_{n}(x) - v^{*}_{n}(x)) \frac{1}{v^{*}_{n}(b_{n})} u^{*}_{m}(b_{n}) = u^{*}_{m}(x) + 0.$$

(vi) We use property (v) and the induction hypothesis as follows: We have $$ v^{*}_{n}(S_{n}(x)) = v^{*}_{n}(T_{n}(S_{n-1}(x))) = u^{*}_{n}(S_{n-1}(x)) = u^{*}_{n}(x).$$  For $ 1 \leq m \leq n - 1 $, we  obtain   $$v^{*}_{m}(S_{n}(x)) = v^{*}_{m}(T_{n}(S_{n-1}(x))) = v^{*}_{m}(S_{n-1}(x)) = u^{*}_{m}(x).$$ Finally, for $ m \geq n + 1 $, we compute $$ u^{*}_{m}(S_{n}(x)) = u^{*}_{m}(T_{n}(S_{n-1}(x))) = u^{*}_{m}(S_{n-1}(x)) = u^{*}_{m}(x). $$

Now, we see from (iv) that $ (S_{n})_n $ is a Cauchy sequence in $ \mathcal{L}(c_{0}) $, having a limit $ S $ with $$ \Vert S - I_{c_{0}} \Vert = \Vert S - S_{0} \Vert \leq \sum_{n=1}^{\infty} 2\varepsilon_{n} \leq \varepsilon.$$ Since we assumed  that $ \varepsilon < 1 $, we obtain that $ S $ is invertible. Also,
$$ v^{*}_{m}(S(x)) = u^{*}_{m}(x), \quad m \in \mathbb{N}, \, x \in c_{0}, $$
as we have $ v^{*}_{m}(S_{n}(x)) = u^{*}_{m}(x) $ for each $ n \geq m $ due to (vi). We finally deduce that
$$ S(E) = \bigcap_{n=1}^{\infty} \{ x \in c_{0} : v^{*}_{n}(x) = 0 \}, $$
which shows that $ S(E) $ is closed in the pointwise topology.

So, we proved our claim  for the case that  $ E $ has infinite codimension. In the opposite case, we can use the same method, with the difference that the sequence $ A_{0} \subset A_{1} \subset A_{2} \subset \dots $ stops after finitely many steps, and the last isomorphism $ S_{n} $ from our recursion works.
\end{proof}

\begin{proof}[Proof of Corollary \ref{coo}]
Let $E$ be a closed infinite-dimensional subspace of the Banach space $c_0$ which is not isomorphic to $c_0$, see \cite[p.~73]{Lindenstrauss}. Let $T(E)$ be the subspace of $c_0$ which is closed in the  pointwise topology of $c_0$ as mentioned in Theorem \ref{co}.   Denote by $T(E)_p$ this space endowed with the pointwise topology. We show that $T(E)_p$ is not isomorphic to $(c_0)_p$. Assume (by contradiction) that there exists an isomorphism $Q: T(E)_p\rightarrow (c_0)_p$. Then, applying the closed graph theorem between Banach spaces $T(E)$ and $c_0$, we derive that $Q: T(E)\rightarrow c_0$ is an isomorphism, which provides a contradiction.
\end{proof}

\end{document}